\documentclass[amssymb,twoside,12pt]{article}
\usepackage{latexsym,amsmath,graphicx}
\usepackage{mathrsfs}
\usepackage{amsfonts}
\newtheorem{theorem}{Theorem}[section]

\newtheorem{definition}{Definition}[section]
\numberwithin{equation}{section}
\newenvironment{proof}{\indent{\em Proof:}}{\quad \hfill
$\Box$\vspace*{2ex}}
\newcommand{\pp}{{\mathfrak P}}

\allowdisplaybreaks
\begin{document}
\setcounter{page}{1}
\begin{center}
\vspace{0.4cm} {\large{\bf  On Impulsive Delay Integrodifferential Equations\\
 with Integral Impulses }} \\
 \vspace{0.5cm}
 Kishor D. Kucche $^{1}$\\
  kdkucche@gmail.com\\

 \vspace{0.35cm}
 Pallavi U. Shikhare $^{2}$ \\
  jananishikhare13@gmail.com \\
 
 \vspace{0.35cm}
 $^{1,2}$ Department of Mathematics, Shivaji University, Kolhapur-416 004, Maharashtra, India.

\end{center}

\def\baselinestretch{1.0}\small\normalsize
\begin{abstract}
The present research paper is devoted to investigate the existence, uniqueness of mild solutions for impulsive delay integrodifferential equations with integral impulses in Banach spaces. We also investigate the dependence of solutions on initial conditions, parameters and the functions involved in the equations. Our analysis is based on semigroup theory, fixed point technique and an application of  Pachpatte’s type integral inequality with integral impulses. An example is provided in support of existence result.
\end{abstract}
\def\baselinestretch{1.5}

\section{Introduction}
Numerous real world physical phenomena can't be modeled as differential equations with classical initial conditions, specially the issues in the part of biology and medical sciences, for instances, the dynamics of populations subject to sudden changes (e.g. diseases, harvesting and so forth.) and few problems in the field of physics, chemical technology, population dynamics, medicine, mechanics, biotechnology and economics. But, such physical phenomena can be modeled by means of differential equations with impulsive conditions. It very well may be noticed that the impulsive conditions are the blends of  classical initial conditions and the short-term perturbations whose span can be insignificant in correlation with the span of the procedure. Therefore  the differential and integrodifferential equations with impulse effect have importance in the modeling many physical phenomena in which sudden changes occurs. The details relating to the theory and applications of  impulsive differential equations (IDEs) and its  impact in further development can be found in the monograph by Bainov and Simeonov \cite{Bainov1}
, Lakshmikantham et al. \cite{Bainov2}, and Samoilenko and Perestyuk \cite{Samoilenko} and in the interesting papers  \cite{Wang,Frigon,Xiaodi,Yan,Bashir,Donal,Liu,Song,Benchohra1,Benchohra2}.

In the begin, Frigon and O'regan \cite{Frigon}, using the fixed point approach determined variety of existence principles for the IDE
\begin{equation} \label{imp0}
\begin{cases}
w'(t)=f\left(t, w(t) \right),\, 0< t <b,\,~t \not= t_{k},\\
w(t_{k}^+)= I_{k}\left( w(t_{k}^-)\right),\,~k=1,2,\cdots,m,
\end{cases}
\end{equation}
with the initial condition $w(0)  =w_{0}$, where $ x(t^{+}_{k})= \underset{\epsilon\to 0^{+}}{\lim} x(t_{k}+\epsilon)$ and $ x(t^{-}_{k})= \underset{\epsilon\to 0^{-}}{\lim} x(t_{k}+\epsilon)$. In addition, utilizing the idea of upper and lower solutions, authors have derived existence results for  the IDE \eqref{imp0} with boundary condition $w(0)=w(b)$. First,  Wang et al. \cite{Wang}  presented the idea of Ulam's type stabilities for  ordinary IDE and derived these stability results by means of integral inequality of Gronwall type for piecewise continuous functions.
Liu \cite{Liu1} have extended the theory to nonlinear impulsive evolution equation \begin{equation} \label{imp1}
\begin{cases}
w'(t)=\mathscr{A}w(t)+f\left(t, w(t) \right),\, 0< t <b,\,~t \not= t_{k},\\
\Delta w(t_{k})= I_{k}\left( w(t_{k})\right),\, ~,\,~k=1,2,\cdots,m, \\
w(0)  =w_{0} , 
\end{cases}
\end{equation}
in a Banach space, where $\mathscr{A}:D(A)\subseteq X \to X$ is the infinitesimal generator of a strongly continuous semigroup of bounded linear operators and applying the semigroup theory explored  the existence and uniqueness of the mild solutions by utilizing fixed point argument.  Further, the author have demonstrated that if f is continuously  differentiable then the mild solutions give rise to classical solution.
Hernández, D. O'regan \cite{Donal}  analyzed a class of abstract impulsive differential equations of the form \eqref{imp1} considering extraordinary sorts of condition on the function $I_{k}$, so as it would allow to think about the impulsive  partial differential equations.

Utilizing the Lerey-Schauder nonlinear alternative, Arjunan and Anguraj \cite{Arjunan} derived  existence of mild solutions  for the impulsive delay differential equations. 
\begin{equation} \label{imp2}
\begin{cases}
w'(t)=\mathscr{A}w(t)+f\left(t, w_{t} \right),\, t \in I=[0,b],~t \not= t_{k},~k=1,2,\cdots,m, \\
w(t)  =\phi(t),\, t\in [-r, 0], \, 0<r<\infty , \\
\Delta w(t_{k})= I_{k}\left( w(t_{k})\right),\, ~k=1,2,\cdots,m
\end{cases}
\end{equation}
Benchora et al. \cite{Benchohra} have extended the above study to  
first and second order impulsive semilinear neutral functional differential equations. 

Thiramanus et al. \cite{Thiramanus} built up the impulsive differential
and impulsive integral inequalities with integral jump conditions  and gave its applications to examine the boundedness and other properties of IDEs with integral jump conditions. Thaiprayoon et al. \cite{Thaiprayoon1,Thaiprayoon2} exhibited the  maximum principles for second-order impulsive integrodifferential equations with integral jump conditions and also developed the monotone iterative technique for a periodic boundary value problem. More subtleties relating to  integral jump conditions can be found in \cite{Thiramanus,Thaiprayoon1,Thaiprayoon2}.

It is seen that the theory of impulsive differential equations have been significant progress. However, very few work have  been accounted  relating to the 
theory of  impulsive delay differential equations with integral impulses. 
Motivated by the previously mentioned work and the applicability of  impulsive differential equation in the scientific community it is imperative to examining the existence, uniqueness, stability and data dependence of solution for broad class of impulsive integrodifferential equations. 

Motivated by the work of \cite{Thiramanus,Thaiprayoon1,Thaiprayoon2}, in the present paper, we consider  the impulsive integrodifferential equation with integral impulses and finite delay described in the form:
\begin{small}
\begin{align} \label{e1.1}
& w'(t)=\mathscr{A} w(t)+V\left(t, w_t, \int_0^t U(t,s,w_s)ds \right), ~t \in I=[0,b],~t \not= t_{k},~k=1,2,\cdots,m, \\
& w(t) =\varsigma(t), ~t \in [-r, 0],\label{e1.2}\\
& \Delta w(t_{k})= I_{k}\left( \int_{t_{k}-\tau_{k}}^{t_{k}-\theta_{k}}G(s,w_s)ds\right),~k=1,2,\cdots,m\label{e1.3}.
\end{align}
\end{small}
where $\mathscr{A}:D(A)\subseteq X \to X$ is the infinitesimal generator of a strongly continuous semigroup of bounded linear operators $\{\mathscr{T}(t)\}_{t\geq 0}$ in the Banach $\left(X, \left\| \cdot\right\|\right)$;  $V: I\times C([-r,0],X)\times X \to X, ~U: I\times I\times C([-r,0],X) \to X $ and $G:J \times X\to X $ are
given functions, $\varsigma\in  C([-r,0],X),(0<r<\infty), 0=t_{0}<t_{1}<\cdots<t_{m}<t_{m+1}= b,\, I_{k}\in C(X,X)$ are bounded functions, $ 0\leq \theta_{k} \leq \tau_{k} \leq t_{k}-t_{k-1} $ for $ k=1,2,\cdots,m$,  $\Delta w(t_{k})= w(t^{+})-w(t^{-}_{k})$,  $ w(t^{+}_{k})= \underset{\epsilon\to 0^{+}}{\lim} w(t_{k}+\epsilon)$ and $ w(t^{-}_{k})= \underset{\epsilon\to 0^{-}}{\lim} w(t_{k}+\epsilon)$. For any continuous function $w$ defined on $[-r,b]-\left\lbrace t_{1},\cdots,t_{m}\right\rbrace $ and  $t\in[0,b]$, denote by $ w_t$ the element of $ C([-r,0],X)$ defined by $w_t(\theta)=w(t+\theta), ~\theta \in[-r,0]$.

Our fundamental point is to research the existence, uniqueness of mild solutions for the problem \eqref{e1.1}--\eqref{e1.3} and examine the dependence  of mild solution on initial conditions, parameters and on the functions involved in the right hand side of the equations \eqref{e1.1}--\eqref{e1.3}. The innovative idea  in the present paper is the utilization of the Pachpatte’s type integral inequality with integral impulses. The innovative thought  to be noted here that  all the qualitative properties mentioned above of the mild solutions for the  problem \eqref{e1.1}--\eqref{e1.3} have been acquired by means of Pachpatte’s type integral inequality with integral impulses for broad class of impulsive Volterra integrodifferential equations just with Lipschitz conditions on the functions $V$and $U$.  

For more uses of the Pachpatte's integral inequalities in the investigation of qualitative properties of solutions  and different kind of Ulam's types stabilities  of general form of integrodifferential equations we refer to the work of Pachpatte \cite{Pachpatte1,Pachpatte} and Kucche et al. \cite{Kucche1,Kucche2}.

The paper is organized as follows.  In section $2$, we introduce notations and state theorems which are useful in further analysis. Section $3$ deals with existence and uniqueness results pertaining to the problem \eqref{e1.1}--\eqref{e1.3}. In section $4$ we establish the dependence of solution on different data. In section 5, an example is given to illustrate the existence and uniqueness results.
\section{Preliminaries} \label{preliminaries}
In this section, we introduce notations and few theorems that are used throughout this paper.
Let $(X,\left\|\cdot\right\|)$ be a Banach space. Then
$$\mathscr{C} =C([-r,0],X)=\left\lbrace \varsigma~:~\varsigma:[-r,0]\to X ~\text{is continuous} \right\rbrace$$
is a Banach space with respect to the norm
\begin{small}
$$\left\|\varsigma \right\|= \sup\left\lbrace\left\|\varsigma(\theta)\right\|: -r<\theta<0\right\rbrace .$$
\end{small}
Further, $C(I,X)=\left\lbrace w~:~w: I\to X ~\text{is continuous} \right\rbrace $ is a Banach space with respect to the norm
\begin{small}
 $$\left\|w\right\|= \sup \left\lbrace\left\|w(t)\right\|: 0<t<b\right\rbrace .$$
 \end{small}
\begin{theorem} [\cite{{Pazy}}] \label{th2.1}
Let $\{\mathscr{T}(t)\}_{t\geq 0}$  is a $C_0$-semigroup. There exists constants $\omega\geq 0$ and $ M\geq 1$
such that $$\|\mathscr{T}(t)\|\leq Me^{\omega t},   ~0\leq t < \infty . $$
\end{theorem}
We utilize the following fixed point theorem due to Burton and Kirk \cite{Kirk}.
\begin{theorem} [\cite{{Kirk}}] \label{th2.3}
Let $\mathcal{X}$ be a Banach space and $\mathcal {A,\,B}: \mathcal X \to \mathcal{X}$ two operators satisfying:
\begin{itemize}
\item[(i)] $\mathcal {A}$ is a contraction, and
\item[(ii)] $ \mathcal {B} $ is completely continuous.
\end{itemize}
Then either
\begin{itemize}
\item[(a)] the operator equation $w=  \mathcal {A}(w)+ \mathcal {B}(w)$ has a solution, or
\item[(b)] the set
\begin{small}
$$\mathcal{Z}=\left\lbrace w \in  \mathcal {X} : w= \lambda  \mathcal {A} \left( \frac{w}{\lambda}\right) +\lambda  \mathcal {B}(w) \right\rbrace$$
\end{small}is unbounded for $\lambda \in (0, 1)$.
\end{itemize}
\end{theorem}
The qualitative analysis of the solution of equations \eqref{e1.1}--\eqref{e1.3} has been done by utilizing the following
 Pachpatte's type integral inequality.
\begin{theorem} [\cite{{Tariboon}}] \label{th2.2}
Let the sequence $\left\lbrace t_{k}\right\rbrace ^{\infty}_{k=1}$ satisfies $0\leq t_{0}<t_{1}<\cdots ~\text{and}~\underset{t\rightarrow \infty}{\lim} ~t_{k}=\infty.$ Let $u \in PC^{1} \left(\mathbb{R}_{+},\mathbb{R}_{+} \right) $ and $u(t)$ is left-continuous at $t_{k},\, k=1,2,\cdots$. Assume that $f,~g$ are nonnegative  continuous functions defined on $ \mathbb{R}_+$ and $ n(t)$ is a positive and nondecreasing
continuous function defined on $ \mathbb{R}_+$.
In addition, assume that constants $\beta_{k}\geq 0,~0\leq \theta_{k} \leq \tau_{k} \leq t_{k}-t_{k-1} $ for $k=1,2,\cdots $.\\
\begin{small}
If
\begin{align*}
u(t)&\leq n(t)+\int_0^t f(s)u(s) ds+\int_0^t f(s) \left(\int_0^s
g(\sigma)u(\sigma)d\sigma\right)ds+\underset{0<t_{k}<t}{\sum} ~\beta_{k} \int_{t_{k}-\tau_{k}}^{t_{k}-\theta_{k}} u(s) ds,\quad t\geq 0,
\end{align*}
\end{small}
then
\begin{small}
\begin{align*}
u(t)&\leq n(t) \underset{0<t_{k}<t}{\prod} ~C_{k} \exp \left (\int_{t_{\alpha}}^t f(s)\left[1+\int_0^s g(\sigma)d\sigma\right] ds\right),\quad t\geq 0,
\end{align*}
\end{small}
where
\begin{small}
\begin{align*}
\alpha&=\max \left\lbrace k~:~ t\geq t_{k},~k=1,2,\cdots\right\rbrace
\end{align*}
\end{small}
and
\begin{small}
\begin{align*}
C_{k}&=\exp \left (\int_{t_{k-1}}^{t_{k}} f(s)\left[1+ \int_0^s g(\sigma)d\sigma\right] ds\right)+\beta_{k} \int_{t_{k}-\tau_{k}}^{t_{k}-\theta_{k}} \exp \left (\int_{t_{k-1}}^{s} f(\sigma)\left[1+ \int_0^\sigma g(\xi)d\xi\right]d\sigma\right)ds.
\end{align*}
\end{small}
\end{theorem}
\section{Existence and Uniqueness Results}

To define the concept of mild solution of \eqref{e1.1}--\eqref{e1.3}, we consider the following space
\begin{align*}
&\varSigma =\left\lbrace  w:[-r,b]\to X :w_{k}\in C(I_{k},X),~k=0,\cdots,m ~\text{and there exist}~w(t^{-}_{k})~\text{and}\right.\\
& \left. w(t^{+}_{k}),~\text{with} ~w(t^{-}_{k}) = w(t_{k}),~k =1, \cdots , m,~w(t)=\varsigma(t),~\forall ~t \in [-r,0]\right\rbrace.
\end{align*}
Define $ \left\| \cdot\right\| _{\varSigma }: \varSigma\to [0, \infty) $
by
\begin{small}
$$\left\|w\right\|_{\varSigma }= \max\left\lbrace\left\|w_{k}\right\|_{I_{k}}: ~k=0\cdots m \right\rbrace ,$$
\end{small}
where $ w_{k}$ is the restriction of $w$ to $I_{k} = [t_{k},  ~t_{k+1}],~k = 0,\cdots,m$. Then $(\varSigma, \left\| \cdot\right\| _{\varSigma } )$ is a Banach space.

Let $ w\in \varSigma $ is a solution of problem the  \eqref{e1.1}--\eqref{e1.3}. Then comparing with the abstract Cauchy problem \cite{{Pazy}}, we have
\begin{small}
$$ w(t)= \mathscr{T}(t)\varsigma(0)+\int_0^t \mathscr{T}(t-s) V\left(s, w_s, \int_0^s U(s,\sigma,w_{\sigma})d\sigma \right)ds,~t \in [0,t_{1}).$$
\end{small}
This gives
\begin{small}
\begin{align}\label{e.k}
w(t^{-}_{1}) = \mathscr{T}(t_{1})\varsigma(0)+\int_0^{t_{1}} \mathscr{T}(t_{1}-s) V\left(s, w_s, \int_0^s U(s,\sigma,w_{\sigma})d\sigma \right)ds.
\end{align}
\end{small}
Now, by condition \eqref{e1.3}, we have
\begin{small}
$$ w(t^{+}_{1})= w(t^{-}_{1})+ I_{1}\left( \int_{t_{1}-\tau_{1}}^{t_{1}-\theta_{1}}G(s,w_s)ds\right).$$
\end{small}
Thus again by comparing with abstract Cauchy problem \cite{{Pazy}}, and using semigroup property for $ t\in (t_{1},t_{2})$ we have
\begin{small}
\begin{align*}
 w(t)& = \mathscr{T}(t-t_{1}) w(t^{+}_{1}) +\int_{t_{1}}^t \mathscr{T}(t-s) V\left(s, w_s, \int_0^s U(s,\sigma,w_{\sigma})d\sigma \right)ds\\
&=\mathscr{T}(t-t_{1})\left[ w(t^{-}_{1})+ I_{1}\left( \int_{t_{1}-\tau_{1}}^{t_{1}-\theta_{1}}G(s,w_s)ds\right)\right]\\
& \qquad + \int_{t_{1}}^t \mathscr{T}(t-s) V\left(s, w_s, \int_0^s U(s,\sigma,w_{\sigma})d\sigma \right)ds\\
&=\mathscr{T}(t-t_{1})\mathscr{T}(t_{1})\varsigma(0)+\int_0^{t_{1}} \mathscr{T}(t-t_{1}) \mathscr{T}(t_{1}-s) V\left(s, w_s, \int_0^s U(s,\sigma,w_{\sigma})d\sigma \right)ds\\
& \qquad+\mathscr{T}(t-t_{1})  I_{1}\left( \int_{t_{1}-\tau_{1}}^{t_{1}-\theta_{1}}G(s,w_s)ds\right)+\int_{t_{1}}^t \mathscr{T}(t-s) V\left(s, w_s, \int_0^s U(s,\sigma,w_{\sigma})d\sigma \right)ds\\
&=\mathscr{T}(t)\varsigma(0)+\int_0^t \mathscr{T}(t-s) V\left(s, w_s, \int_0^s U(s,\sigma,w_{\sigma})d\sigma \right)ds+\mathscr{T}(t-t_{1})  I_{1}\left( \int_{t_{1}-\tau_{1}}^{t_{1}-\theta_{1}}G(s,w_s)ds\right).
\end{align*}
\end{small}
Reiterating these procedures, one can prove that
\begin{small}
\begin{align*}
w(t)&=\mathscr{T}(t)\varsigma(0)+\int_0^t \mathscr{T}(t-s) V\left(s, w_s, \int_0^s U(s,\sigma,w_{\sigma})d\sigma \right)ds\\
&\qquad+\underset{0<t_{k}<t}{\sum} \mathscr{T}(t-t_{k}) I_{k}\left( \int_{t_{k}-\tau_{k}}^{t_{k}-\theta_{k}}G(s,w_s)ds\right),~ t\in I.
\end{align*}
\end{small}
\begin{definition}
A function $w:[-r,b]\to X $ is called a mild solution of \eqref{e1.1}--\eqref{e1.3} if $~w(t)=\varsigma(t),\text{on}\,[-r,0],$  and
\begin{small}
\begin{align*}
w(t)&=\mathscr{T}(t)\varsigma(0)+\int_0^t \mathscr{T}(t-s) V\left(s, w_s, \int_0^s U(s,\sigma,w_{\sigma})d\sigma \right)ds\\
&\qquad+\underset{0<t_{k}<t}{\sum} \mathscr{T}(t-t_{k}) I_{k}\left( \int_{t_{k}-\tau_{k}}^{t_{k}-\theta_{k}}G(s,w_s)ds\right),~ t\in I.
\end{align*}
\end{small}
\end{definition}
To obtain existence and uniqueness results pertaining to the problem  \eqref{e1.1}--\eqref{e1.3}, we need the following conditions:
\begin{itemize}\item[(H1)]
There exists functions $N_{V}, N_{U}\in C(I,\mathbb{R}_+)$ and the constant $L_{G}>0$  such that
\begin{small}
\begin{itemize}
\item[(i)]$\|V(t,\vartheta_{1},w_{1})-V(t,\vartheta_{2},w_{2})\| \leq N_{V}(t) \left(\|\vartheta_{1}-\vartheta_{2}\|_{\mathscr{C}}+\|w_{1}-w_{2}\|\right);$
\item[(ii)]$ \|U(t,s,\vartheta_{1})-U(t,s,\vartheta_{2})\|\leq N_{U}(t) \,  \|\vartheta_{1}-\vartheta_{2}\|_{\mathscr{C}};$
\item[(iii)] $\|G(t,\vartheta_{1})-G(t,\vartheta_{2})\| \leq L_{G} \,  \|\vartheta_{1}-\vartheta_{2}\|_{\mathscr{C}};$
\end{itemize}
\end{small}
for all $ t,s \in I,~\vartheta_{1},\vartheta_{2}\in \mathscr{C}$ and $w_{1},w_{2}\in X .$
\item[(H2)]There exist constants $\mathscr{D}_{k} (k= 1, \cdots, m) $ such that $\left\| I_{k}(w_{1})-I_{k}(w_{2})\right\| \leq \mathscr{D}_{k}\left\|w_{1}-w_{2}\right\|;~ k= 1, \cdots, m$ for $w_{1},w_{2}\in X$.
\end{itemize}
\begin{theorem}\label{th3.3}
(Existence) Assume that the hypotheses (H1)--(H2) are satisfied and let
\begin{small}
\begin{align} \label{c1}
\sum_{k=1}^{m}2b\, \mathscr {M} L_{G}\,\mathscr{D}_{k} < 1 .
\end{align}
\end{small} Then
the problem \eqref{e1.1}--\eqref{e1.3} has at least one solution in $\varSigma$.
\end{theorem}
\begin{proof}
Transform the problem \eqref{e1.1}--\eqref{e1.3} into a fixed point problem consider the operators $\mathcal{A},\mathcal{B}:\varSigma \to \varSigma$ defined by
\begin{small}
\begin{equation*}
\mathcal{A}(w)(t) =
\begin{cases}
0 & \text{if $t \in [-r,0], $}\\
\underset{0<t_{k}<t}{\sum} \mathscr{T}(t-t_{k}) I_{k}\left( \int_{t_{k}-\tau_{k}}^{t_{k}-\theta_{k}}G(s,w_s)ds\right) & \text{if $t \in [0,b] $}.
\end{cases}
\end{equation*}
\end{small}
and
\begin{small}
\begin{equation*}
\mathcal{B}(w)(t) =
\begin{cases}
\varsigma(t) & \text{if $t \in [-r,0], $}\\
\mathscr{T}(t)\varsigma(0)+\int_0^t \mathscr{T}(t-s) V\left(s, w_s, \int_0^s U(s,\sigma,w_{\sigma})d\sigma \right) ds & \text{if $t \in [0,b] $}.
\end{cases}
\end{equation*}
\end{small}
In the view of the Theorem \ref{th2.3} to prove the operator equation $w=  \mathcal {A}(w)+ \mathcal {B}(w)$ have fixed point, we prove the following statements.
\begin{small}
\begin{itemize}
\item[(i)] $\mathcal {B}: \varSigma \to \varSigma$ is completely continuous.
\item[(ii)]  $\mathcal {A}: \varSigma \to \varSigma$ is contraction.
\item[(iii)] the set  $$\mathcal{Z}=\left\lbrace w \in  \varSigma : w= \lambda  \mathcal {A} \left( \frac{w}{\lambda}\right) +\lambda  \mathcal {B}(w) \right\rbrace$$
is bounded for some $ \lambda \in (0,1)$.
\end{itemize}
\end{small}
The proof of the operator $\mathcal {B}$ is completely continuous is given in the following steps.\\
Step 1: $\mathcal{B} : \varSigma \to \varSigma$ is continuous .

Let any $w\in \varSigma $ and $\left\lbrace w_{n}\right\rbrace ^{\infty}_{n=1}$ be any sequence in $ \varSigma $ such that $w_{n}\to w$. Then $ \left\|w_{n}\right\|_{\varSigma} \leq q^{'} ,\, \forall\, n\in \mathbb{N} $ and $ q^{'}>0 $ and hence $ \left\|w\right\|_{\varSigma} \leq q^{'}$. Thus we have $\left\lbrace w_{n}\right\rbrace \subseteq B_{q^{'}} $ and $w\, \epsilon\, B_{q^{'}} $. Now by  Lebesgue dominated convergence theorem,
\begin{small}
\begin{align*}
 \left\| \mathcal{B}(w_{n})-\mathcal{B}(w)\right\| _{\varSigma}
&\leq\underset{[-r, b]}{\sup}\left\lbrace \int_0^t \left\|  \mathscr{T}(t-s)\right\| _{B(X)} \left\| V\left( s, w_{n_{s}},\int_0^s U\left( s,\sigma ,w_{n_{\sigma}}\right) d\sigma\right)\right.\right.\\
&\qquad \left. \left.-V\left( s, w_{s},\int_0^s U\left( s,\sigma ,w_{\sigma}\right) d\sigma\right)\right\| ds\right\rbrace \to 0 \,\text{as}\, n \to \infty.
\end{align*}
\end{small}
Thus $\mathcal{B}$ is continuous.

Step 2: $\mathcal {B}$ maps bounded sets into bounded sets in $\varSigma$.

Let $ w\in B_{q}=\left\lbrace w\in \varSigma:\left\| w\right\|_{\varSigma}\leq q\right\rbrace$. It is sufficient to prove that there exist $l>0$ such that $ \left\| \mathcal {B}(w)\right\|_{\varSigma}\leq l,\, \forall\, w \in B_{q}$. Now for any $t\in [0,b]$,
\begin{small}
\begin{align}\label{b1}
& \left\| \mathcal {B}(w)(t)\right\|\nonumber\\
&\quad\leq \left\| \mathscr{T}(t)\right\| _{B(X)} \left\|  \varsigma(0)\right\|+\int_0^t \left\|  \mathscr{T}(t-s)\right\| _{B(X)} \left\| V\left( s, w_{s},\int_0^s U\left( s,\sigma ,w_{\sigma}\right) d\sigma\right) \right.\nonumber\\
& \qquad \left. -V\left( s, 0,\int_0^s U\left( s,\sigma, 0\right) d\sigma\right)\right\| ds +\int_0^t \left\|  \mathscr{T}(t-s)\right\| _{B(X)} \left\|  V\left( s, 0,\int_0^s U\left( s,\sigma, 0\right) d\sigma\right)\right\| ds.
\end{align}
\end{small}
By the Theorem \ref{th2.1} there exist constant $\mathscr {M}\geq 1 $ such that
\begin{align}
\left\| \mathscr{T}(t)\right\| _{B(X)} \leq \mathscr {M},\, t\geq 0. \label{e0.3}
\end{align}
Thus from inequality \eqref{b1} and by hypothesis (H1)(i) and (H1)(ii) we have
\begin{small}
\begin{align*}
& \left\| \mathcal {B}(w)(t)\right\|\\
&\quad\leq \mathscr {M} \left\|  \varsigma\right\|_{\mathscr {C}}+\int_0^t \mathscr {M} N_{V}(s) \left\lbrace  \left\|w_{s} \right\|_{\mathscr {C}}+ \int_0^s N_{U}(\sigma) \left\|w_{\sigma} \right\|_{\mathscr {C}} d\sigma \right\rbrace ds\\
&\qquad+\int_0^t \mathscr {M}\left\|  V\left( s, 0,\int_0^s U\left( s,\sigma, 0\right) d\sigma\right)\right\| ds\\
&\quad\leq \mathscr {M} \left\|  \varsigma\right\|_{\mathscr {C}}+\int_0^b \mathscr {M} N_{V}(s) \left\lbrace  \left\|w_{s} \right\|_{\mathscr {C}}+ \int_0^s N_{U}(\sigma) \left\|w_{\sigma} \right\|_{\mathscr {C}} d\sigma \right\rbrace ds\\
&\qquad+\int_0^b \mathscr {M}\left\|  V\left( s, 0,\int_0^s U\left( s,\sigma, 0\right) d\sigma\right)\right\| ds\\
&\quad\leq \mathscr {M} \left\|  \varsigma\right\|_{\mathscr {C}}+\int_0^b q \mathscr{M} N_{V}(s) \left\lbrace 1 + \int_0^s N_{U}(\sigma) d\sigma \right\rbrace ds+\int_0^b \mathscr {M}\left\|  V\left( s, 0,\int_0^s U\left( s,\sigma, 0\right) d\sigma\right)\right\| ds\, =l^{*}.
\end{align*}
\end{small}
Further, for any $t\in [-r,0]$ we have,
$\left\| \mathcal {B}(w)(t)\right\| =  \left\| \varsigma(t)\right\|
 \leq \left\| \varsigma\right\|_{\mathscr {C}}$. Thus
$$\left\| \mathcal {B}\,w\right\|_{\varSigma}= \underset{[-r,b]}{\sup} \left\| \mathcal {B}(w)(t)\right\|
 \leq \max\left\lbrace \left\| \varsigma\right\|_{\mathscr {C}},\,l^{*}\right\rbrace =l.$$
 Step 3: $\mathcal{B}$ maps bounded sets into equicontinuous sets of $\varSigma$.

Let $ B_{q}=\left\lbrace w\in \varSigma:\left\| w\right\|_{\varSigma}\leq q\right\rbrace$ be any bounded set of $\varSigma$. Then utilizing hypothesis (H1)(i) and (H1)(ii) for any $\xi_{1},\xi_{2}\in [0,b]$ with $\xi_{1}<\xi_{2}$ we have,
\begin{small}
\begin{align*}
&\left\| \mathcal{B}(w)(\xi_{2})-\mathcal{B}(w)(\xi_{1})\right\|\\
&\leq \left\| \mathscr{T}(\xi_{1})-\mathscr{T}(\xi_{2})\right\| _{B(X)} \left\|  \varsigma(0)\right\|+\int_0^{\xi_{2}} \left\|  \mathscr{T}(\xi_{2}-s)\right\| _{B(X)} \left\| V\left( s, w_{s},\int_0^s U\left( s,\sigma ,w_{\sigma}\right) d\sigma\right) \right\| ds\\
& \qquad -\int_0^{\xi_{1}} \left\|  \mathscr{T}(\xi_{1}-s)\right\| _{B(X)} \left\| V\left( s, w_{s},\int_0^s U\left( s,\sigma ,w_{\sigma}\right) d\sigma\right) \right\| ds\\
&= \left\| \mathscr{T}(\xi_{1})-\mathscr{T}(\xi_{2})\right\| _{B(X)} \left\|  \varsigma(0)\right\|+\int_0^{\xi_{1}} \left\|  \mathscr{T}(\xi_{2}-s)-\mathscr{T}(\xi_{1}-s)\right\| _{B(X)} \left\| V\left( s, w_{s},\int_0^s U\left( s,\sigma ,w_{\sigma}\right) d\sigma\right) \right\| ds\\
& \qquad +\int_{\xi_{1}}^{\xi_{2}} \left\|  \mathscr{T}(\xi_{2}-s)\right\| _{B(X)} \left\| V\left( s, w_{s},\int_0^s U\left( s,\sigma ,w_{\sigma}\right) d\sigma\right) \right\| ds\\
& \leq\left\| \mathscr{T}(\xi_{1})-\mathscr{T}(\xi_{2})\right\| _{B(X)} \left\|  \varsigma(0)\right\|+\int_0^{\xi_{1}} \left\|  \mathscr{T}(\xi_{2}-s)-\mathscr{T}(\xi_{1}-s)\right\| _{B(X)}\\
& \quad\left\| V\left( s, w_{s},\int_0^s U\left( s,\sigma ,w_{\sigma}\right) d\sigma\right) - V\left( s, 0,\int_0^s U\left( s,\sigma, 0\right) d\sigma\right)\right\| ds\\
& \quad +\int_0^{\xi_{1}} \left\| \mathscr{T}(\xi_{1})-\mathscr{T}(\xi_{2})\right\| _{B(X)} \left\| V\left( s, 0,\int_0^s U\left( s,\sigma, 0\right) d\sigma\right)\right\|ds\\
& \qquad +\int_{\xi_{1}}^{\xi_{2}} \left\|  \mathscr{T}(\xi_{2}-s)\right\| _{B(X)} \left\| V\left( s, w_{s},\int_0^s U\left( s,\sigma ,w_{\sigma}\right) d\sigma\right)-V\left( s, 0,\int_0^s U\left( s,\sigma, 0\right) d\sigma\right) \right\| ds\\
& \quad\qquad + \int_{\xi_{1}}^{\xi_{2}} \left\|  \mathscr{T}(\xi_{2}-s)\right\| _{B(X)}\left\| V\left( s, 0,\int_0^s U\left( s,\sigma, 0\right) d\sigma\right) \right\|ds\\
& \leq\left\| \mathscr{T}(\xi_{1})-\mathscr{T}(\xi_{2})\right\| _{B(X)} \left\|  \varsigma\right\|_{\mathscr C}+\int_0^{\xi_{1}} \left\|  \mathscr{T}(\xi_{2}-s)-\mathscr{T}(\xi_{1}-s)\right\| _{B(X)}
N_{V}(s) \left\lbrace  \left\|w_{s} \right\|_{\mathscr {C}}\right.\\
& \quad \left. + \int_0^s N_{U}(\sigma) \left\|w_{\sigma} \right\|_{\mathscr {C}} d\sigma \right\rbrace ds+\int_0^{\xi_{1}} \left\| \mathscr{T}(\xi_{1})-\mathscr{T}(\xi_{2})\right\| _{B(X)} \left\| V\left( s, 0,\int_0^s U\left( s,\sigma, 0\right) d\sigma\right)\right\|ds\\
& \qquad +\int_{\xi_{1}}^{\xi_{2}} \left\|  \mathscr{T}(\xi_{2}-s)\right\| _{B(X)} N_{V}(s) \left\lbrace  \left\|w_{s} \right\|_{\mathscr {C}}+ \int_0^s N_{U}(\sigma) \left\|w_{\sigma} \right\|_{\mathscr {C}} d\sigma \right\rbrace ds\\
& \quad\qquad + \int_{\xi_{1}}^{\xi_{2}} \left\|  \mathscr{T}(\xi_{2}-s)\right\| _{B(X)}\left\| V\left( s, 0,\int_0^s U\left( s,\sigma, 0\right) d\sigma\right) \right\|ds\\
& \leq\left\| \mathscr{T}(\xi_{1})-\mathscr{T}(\xi_{2})\right\| _{B(X)} \left\|  \varsigma\right\|_{\mathscr C}+\int_0^{\xi_{1}} \left\|  \mathscr{T}(\xi_{2}-s)-\mathscr{T}(\xi_{1}-s)\right\| _{B(X)}
q\,N_{V}(s) \left\lbrace  1 + \int_0^s N_{U}(\sigma)d\sigma \right\rbrace ds\\
& \quad +\int_0^{\xi_{1}} \left\| \mathscr{T}(\xi_{1})-\mathscr{T}(\xi_{2})\right\| _{B(X)} \left\| V\left( s, 0,\int_0^s U\left( s,\sigma, 0\right) d\sigma\right)\right\|ds\\
& \qquad +\int_{\xi_{1}}^{\xi_{2}} \left\|  \mathscr{T}(\xi_{2}-s)\right\| _{B(X)}\,q\, N_{V}(s) \left\lbrace  1+ \int_0^s N_{U}(\sigma)  d\sigma \right\rbrace ds\\
& \quad\qquad + \int_{\xi_{1}}^{\xi_{2}} \left\|  \mathscr{T}(\xi_{2}-s)\right\| _{B(X)}\left\| V\left( s, 0,\int_0^s U\left( s,\sigma, 0\right) d\sigma\right) \right\|ds
\end{align*}
\end{small}
As $\xi_{1}\to\xi_{2} $ the right hand side of the above inequality tends to zero. For the cases $ -r< \xi_{1}<\xi_{2}\leq 0$ and $ -r< \xi_{1}<0<\xi_{2}<b\leq 0$ the equicontinuity one can be obtain easily via uniform continuity of the function $\zeta$ and above case.
 Thus, $ \left\| \mathcal{B}(w)(t_{1})-\mathcal{B}(w)(t_{2})\right\|\leq \aleph |t_{1}-t_{2}|;$ with constant $\aleph >0, \forall \, t\in [-r,b]$ and we conclude that $\mathcal{B}_{B_{q}}$ is an equicontinuous family of functions.

The equicontinuity and uniform boundedness of the  set $\mathcal{B}_{B_{q}}=\left\lbrace (\mathcal{B} w) (t)~:\left\| w\right\|_{\varSigma}\leq q,\, -r \leq r\leq b \right\rbrace$  is as of now illustrated.  In the perspective of Arzela-Ascoli's, if we prove  $$\left\lbrace (\mathcal{B} w) (t)~:\left\| w\right\|_{\varSigma}\leq q,\, -r \leq r\leq b \right\rbrace $$ is
precompact in $\varSigma$  for every $t\in [-r, b]$ then $\mathcal{B}_{B_{q}}$ is precompact in $\varSigma$.

Let $0<t< b$ be fixed and let $\epsilon$ be a real number satisfying $ 0<\epsilon<t$. For $w \in B_{q},$ we define
\begin{small}
$$ (\mathcal{B}_{\epsilon}w)(t)=T(t)\varsigma(0)+T(\epsilon) \int_0^{(t-\epsilon)} T(t-\epsilon-s) V\left( s, w_{s},\int_0^{s} U\left( s,\sigma,w_{\sigma}\right) d\sigma \right)ds.$$
\end{small}
 Since $ T(t)$ is a compact operator, the set
 \begin{small}
 $$ \mathcal {W}_{\epsilon}(t)=\left\lbrace (\mathcal{B}_{\epsilon}w)(t)~: w\in B_{q}\right\rbrace $$
 \end{small}
 is compact in $\varSigma$ for every $\epsilon,\,0<\epsilon<t$. Moreover, for every $w\in B_{q}$ we have
\begin{small}
$$\left\| \mathcal{B}(w)(t)-\mathcal{B}_{\epsilon}(w)(t)\right\|\leq \int_{(t-\epsilon)}^{t} \left\| T(t-s)\right\| _{B(X)} \left\| V\left( s, w_{s},\int_0^{s} U\left( s,\sigma,w_{\sigma}\right) d\sigma \right)\right\| ds. $$
\end{small}
See that the precompact sets are arbitrarily close to the set $\left\lbrace (\mathcal{B}w)(t)~: w\in B_{q}\right\rbrace $ and hence $\left\lbrace (\mathcal{B}w)(t)~: w\in B_{q}\right\rbrace $ is precompact in $\varSigma$. We have proved that the operator $ \mathcal{B}$ is completely continuous.

Next, we prove that the operator $\mathcal{A}$ is a contraction.
Let any $ w_{1}, w_{2} \in \varSigma$. Then by using (H1)(iii) and (H2) then for any $ t\in [0,b]$, we get
\begin{small}
\begin{align*}
&\left\| \mathcal{A}(w_{1})(t)-\mathcal{A}(w_{2})(t)\right\| \\
&\quad\leq \underset{0<t_{k}<t}{\sum} \left\| \mathscr{T}(t-t_{k})\right\|_{B(X)}  \left[ \left\| I_{k}\left( \int_{t_{k}-\tau_{k}}^{t_{k}-\theta_{k}}G(s,(w_{1})_s)ds\right)-I_{k}\left( \int_{t_{k}-\tau_{k}}^{t_{k}-\theta_{k}}G(s,(w_{2})_s)ds\right)\right\| \right] \\
& \quad\leq \sum_{k=1}^{m}\mathscr {M} L_{G}\,\mathscr{D}_{k} \int_{t_{k}-\tau_{k}}^{t_{k}-\theta_{k}}  \left\|(w_{1})_s-(w_{2})_s\right\|_{\mathscr C}  ds\\
&\quad\leq \sum_{k=1}^{m}\mathscr {M} L_{G}\,\mathscr{D}_{k} \int_{t_{k}-\tau_{k}}^{t_{k}-\theta_{k}}  \left\|w_{1}-w_{2}\right\|_{\varSigma} ds\\
&\quad\leq \sum_{k=1}^{m}2b\, \mathscr {M} L_{G}\,\mathscr{D}_{k}  \left\|w_{1}-w_{2}\right\|_{\varSigma}.
\end{align*}
\end{small}

Finally, we prove that the set
\begin{small}
$$ \mathcal{Z}=\left\lbrace w \in \varSigma : w= \lambda \mathcal{A}\left( \frac{w}{\lambda}\right) +\lambda \mathcal{B}(w)\, \text{for some}\, 0<\lambda<1\right\rbrace  $$
\end{small}
 is bounded. Consider the operator equation,
 \begin{small}
 $$ w=  \lambda \mathcal{A}\left( \frac{w}{\lambda}\right) +\lambda \mathcal{B}(w),\,w\in \mathcal{Z}\, \text{\,for some}\,0< \lambda <1.$$
 \end{small}
 Then
\begin{small}
\begin{align*}
w(t)&= \lambda T(t) \varsigma (0)+ \lambda \int_{0}^{t} \mathscr{T}(t-s) V\left(s, w_s, \int_0^s U(s,\sigma,w_{\sigma})d\sigma \right)ds\\
&\qquad+\lambda\underset{0<t_{k}<t}{\sum} \mathscr{T}(t-t_{k}) I_{k}\left( \int_{t_{k}-\tau_{k}}^{t_{k}-\theta_{k}}\frac{G(s,w_s)}{\lambda}ds\right),\, t\in [0,b].
 \end{align*}
 \end{small}
 By using (H1)--(H2), for $\lambda\in (0,1)$ and $ t\in [0,b]$, we have
 \begin{small}
\begin{align*}
&\left\| w(t)\right\|\\
& \leq |\lambda| \left\| T(t)\right\| _{B(X)} \left\| \varsigma\right\| _{\mathscr C} + |\lambda| \int_{0}^{t} \left\| \mathscr{T}(t-s)\right\| _{B(X)}\left\| V\left(s, w_s, \int_0^s U(s,\sigma,w_{\sigma})d\sigma \right) \right.\\
&\qquad \left. -V\left(s, 0, \int_0^s U(s,\sigma,0)d\sigma \right)\right\| ds+ |\lambda| \int_{0}^{t} \left\| \mathscr{T}(t-s)\right\|_{B(X)}\left\| V\left(s, 0, \int_0^s U(s,\sigma,0)d\sigma \right)\right\| ds  \\
&\quad\qquad +|\lambda|\underset{0<t_{k}<t}{\sum} \left\| \mathscr{T}(t-t_{k})\right\|_{B(X)} \left\|  I_{k}\left( \int_{t_{k}-\tau_{k}}^{t_{k}-\theta_{k}}\frac{G(s,w_s)}{\lambda}ds\right) -I_{k}\left( \int_{t_{k}-\tau_{k}}^{t_{k}-\theta_{k}}\frac{G(s,0)}{\lambda} ds\right)\right\|\\
&\qquad \qquad+ |\lambda|\sum_{k=1}^{m}\left\| \mathscr{T}(t-t_{k})\right\|_{B(X)} \left\| I_{k}\left( \int_{t_{k}-\tau_{k}}^{t_{k}-\theta_{k}}\frac{G(s,0)}{\lambda} ds\right)\right\|\\
& \leq \mathscr M \left\| \varsigma\right\| _{\mathscr C}+ \int_0^t M N_{v}(s) \left\lbrace \left\| w_{s}\right\| _{\mathscr C}+\int_0^s\left\| w_{\sigma}\right\| _{\mathscr C} d\sigma\right\rbrace  ds + \int_{0}^{b} \mathscr {M}\left\| V\left(s, 0, \int_0^s U(s,\sigma,0)d\sigma \right)\right\| ds\\
& \qquad +\underset{0<t_{k}<t}{\sum} \mathscr {M} L_{G}\,\mathscr{D}_{k} \int_{t_{k}-\tau_{k}}^{t_{k}-\theta_{k}}  \left\|w_{s}\right\|_{\mathscr C} ds+ \sum_{k=1}^{m} \mathscr {M}\left\| I_{k}\left( \int_{t_{k}-\tau_{k}}^{t_{k}-\theta_{k}}\frac{G(s,0)}{\lambda} ds\right)\right\|.
\end{align*}
 \end{small}
 Denote
 \begin{small}
  $$\mathcal{H} =\int_{0}^{b} \mathscr {M}\,\left\| V\left(s, 0, \int_0^s U(s,\sigma,0)d\sigma \right)\right\| ds\, \text{\,and\,}\,\mathcal{Q}=\sum_{k=1}^{m}  \mathscr {M}\left\| I_{k}\left( \int_{t_{k}-\tau_{k}}^{t_{k}-\theta_{k}}\frac{G(s,0)}{\lambda} ds\right)\right\| .$$
 \end{small} Set $\pp(t)=\sup\{\|w(s)\|:s\in[-r,t]\},~t \in [0,b] $. Observe that $\|w_t\|_\mathscr {C}\leq \pp(t) $  for all $t\in [0,b]$ and there
 is $\eta\in [-r,t]$  such that $ \pp(t)=\|w(\eta)\|$. Hence for $\eta\in [0,t]$ we have
\begin{small}
\begin{align} \label{e}
\pp(t) &=\|w(\eta)\|\nonumber\\
&\leq \mathscr M \left\| \varsigma\right\| _{\mathscr C}+ \mathcal{H}+\mathcal{Q}+ \int_0^{\eta} \mathscr M N_{v}(s) \left\lbrace \left\| w_{s}\right\| _{\mathscr C}+\int_0^s N_{U}(\sigma)\left\| w_{\sigma}\right\| _{\mathscr C} d\sigma\right\rbrace  ds\nonumber\\
&\qquad+\underset{0<t_{k}<t}{\sum} \mathscr {M} L_{G}\,\mathscr{D}_{k} \int_{t_{k}-\tau_{k}}^{t_{k}-\theta_{k}}  \left\|w_{s}\right\|_{\mathscr C} ds\nonumber\\
&\leq \mathscr M \left\| \varsigma\right\| _{\mathscr C}+ \mathcal{H}+\mathcal{Q}+ \int_0^t \mathscr M N_{v}(s) \left\lbrace \pp(s)+\int_0^s N_{U}(\sigma)\pp(\sigma) d\sigma\right\rbrace  ds \nonumber\\ &\qquad+\underset{0<t_{k}<t}{\sum} \mathscr {M} L_{G}\,\mathscr{D}_{k} \int_{t_{k}-\tau_{k}}^{t_{k}-\theta_{k}} \pp(s) ds.
\end{align}
 \end{small}
 If $\eta \in [-r.0]$ then $\pp(s)=\left\| \varsigma(t)\right\|\leq \left\| \varsigma\right\|_{\mathscr C}$. In this case the inequality \eqref{e} holds clearly since  $\mathscr {M}\geq 1 $.
 Applying the Pachpatte’s type integral inequality with integral impulses given in Theorem \ref{th2.2} to \eqref{e} with
  \begin{small}
 $$~u(t)= \pp(t),\,\, n(t)=\mathscr M \left\| \varsigma\right\| _{\mathscr C}+ \mathcal{H}+\mathcal{Q},~ f(t)=\mathscr {M} N_{V}(t) ,~g(t)= N_{U}(t),~\text{and}\, ~\beta_{k}= \mathscr {M} L_{G}\,\mathscr{D}_{k}$$
  \end{small}
 we obtain
 \begin{small}
 \begin{align} \label{e1}
\pp(t)& \leq  \left( \mathscr M \left\| \varsigma\right\| _{\mathscr C} +\mathcal{H}+\mathcal{Q}\right) \underset{0<t_{k}<t}{\prod} ~C_{k} \exp \left (\int_{t_{\alpha}}^t \mathscr M N_{V}(s) \left[1+\int_0^s N_{U}(\sigma) d\sigma\right] ds\right)\nonumber\\
& \leq  \left( \mathscr M \left\| \varsigma\right\| _{\mathscr C} +\mathcal{H}+\mathcal{Q}\right) \underset{0<t_{k}<t}{\prod} ~C_{k} \exp \left (\int_{t_{\alpha}}^b \mathscr M N_{V}(s) \left[1+\int_0^s N_{U}(\sigma) d\sigma\right] ds\right)\nonumber\\
&:=\mathcal{K},\, t\in [0,b],
\end{align}
\end{small}
where
\begin{small}
\begin{align}\label{b4}
C_{k}&=\exp \left (\int_{t_{k-1}}^{t_{k}} \mathscr {M} N_{V}(s)\left[1+ \int_0^s N_{U}(\sigma)d\sigma\right] ds\right)\nonumber\\
&\qquad +\mathscr {M} L_{G}\,\mathscr{D}_{k} \int_{t_{k}-\tau_{k}}^{t_{k}-\theta_{k}} \exp \left (\int_{t_{k-1}}^{s} \mathscr {M} N_{V}(\sigma)\left[1+ \int_0^\sigma N_{U}(\xi)d\xi\right]d\sigma\right)ds.
\end{align}
\end{small}
This gives $ \left\| w\right\| _{\varSigma} \leq \mathcal{K}$. We have proved that the set $\mathcal{Z}$ is bounded. In the view of the Theorem \ref{th2.3}, the operator equation $w=  \mathcal {A}(w)+ \mathcal {B}(w),\, w\in \varSigma $ has a fixed point which is  solution of problem \eqref{e1.1}--\eqref{e1.3}.
\end{proof}
\begin{theorem}(Uniqueness)
If the functions $ V,~U$ and $ G$ satisfies the conditions (H1)--(H2) then  the problem  \eqref{e1.1}--\eqref{e1.3} has unique solution.
\end{theorem}
\begin{proof}
Let $w_{1}$ and $w_{2}$ be the mild solutions of \eqref{e1.1}--\eqref{e1.3} then we have
\begin{small}
\begin{align}\label{3.1}
& w_{1}(t)- w_{2}(t)\nonumber\\
&=\int_0^t \mathscr{T}(t-s) \left[ V\left(s, (w_{1})_s, \int_0^s U(s,\sigma,(w_{1})_{\sigma})d\sigma \right)-V\left(s, (w_{2})_s, \int_0^s U(s,\sigma,(w_{2})_{\sigma})d\sigma \right)\right] ds\nonumber\\
& \qquad+\underset{0<t_{k}<t}{\sum} \mathscr{T}(t-t_{k}) \left[ I_{k}\left( \int_{t_{k}-\tau_{k}}^{t_{k}-\theta_{k}}G(s,(w_{1})_s)\right) ds -I_{k}\left( \int_{t_{k}-\tau_{k}}^{t_{k}-\theta_{k}} G(s,(w_{2})_s)ds\right)\right] ,~ t\in [0,b]
\end{align}
\end{small}
and
\begin{small}
\begin{align}\label{3.2}
w_{1}(t)- w_{2}(t)=0, \quad t\in [-r,0].
\end{align}
\end{small}
Then using the hypotheses (H1) and (H2),
\begin{small}
\begin{align*}
&\left\| w_{1}(t)- w_{2}(t)\right\|\\
&\leq \int_0^t\left\| \mathscr{T}(t-s)\right\|_{B(X)} N_{V}(s)\left\lbrace\left\|(w_{1}- w_{2})(s)\right\| _{\mathscr {C}}
+\int_0^s N_{U}(\sigma)\left\| (w_{1}- w_{2})(\sigma)\right\| _{\mathscr {C}} d\sigma\right\rbrace ds \\
& \qquad +\underset{0<t_{k}<t}{\sum} \left\| \mathscr{T}(t-t_{k})\right\| _{B(X)} L_{G}\,\mathscr{D}_{k} \int_{t_{k}-\tau_{k}}^{t_{k}-\theta_{k}}\left\| (w_{1}- w_{2})(s)\right\| _{\mathscr {C}} ds,\, t \in [0,b].
\end{align*}
\end{small}
Set
\begin{small}
\begin{align}\label{3}
\tilde{\pp}(t)=\sup\{\|(w-v)(s)\|:s\in[-r,t]\},~t \in [0,b] .
\end{align}
\end{small}
 Observe that $\|(w-v)_t\|_\mathscr {C}\leq \tilde{\pp}(t) $  for all $t\in [0,b]$ and there
 is $\eta\in [-r,t]$  such that $ \tilde{\pp}(t)=\|(w-v)(\tilde{\eta})\|$. Hence for $\tilde{\eta}\in [0,t]$ we have
\begin{small}
\begin{align}\label{3.4}
\tilde{\pp}(t)&\leq \int_0^{\tilde{\eta}} \mathscr {M} N_{V}(s)\left\lbrace \tilde{\pp}(s) +\int_0^s N_{U}(\sigma) \tilde{\pp}(\sigma) d\sigma\right\rbrace ds  +\underset{0<t_{k}<t}{\sum} \mathscr {M} L_{G}\,\mathscr{D}_{k} \int_{t_{k}-\tau_{k}}^{t_{k}-\theta_{k}}  \tilde{\pp}(s)ds\nonumber\\
&\leq \int_0^t \mathscr {M} N_{V}(s)\left\lbrace \tilde{\pp}(s) +\int_0^s N_{U}(\sigma) \tilde{\pp}(\sigma) d\sigma\right\rbrace ds  +\underset{0<t_{k}<t}{\sum} \mathscr {M} L_{G}\,\mathscr{D}_{k} \int_{t_{k}-\tau_{k}}^{t_{k}-\theta_{k}}  \tilde{\pp}(s)ds.
\end{align}
\end{small}
 If $\tilde{\eta}\in [-r.0]$ then $\tilde{\pp}(s)=0$ and the inequality \eqref{3.4} holds obviously since $\mathscr {M}\geq 1 $.
Applying the Pachpatte’s type integral inequality given in Theorem \ref{th2.2} to \eqref{3.4} with
\begin{small}
$$~u(t)= \tilde{\pp}(t),\,\, n(t)=0,~ f(t)=\mathscr {M} N_{V}(t) ,~g(t)= N_{U}(t),~\text{and}\, ~\beta_{k}= \mathscr {M} L_{G}\,\mathscr{D}_{k}$$
\end{small}
gives that $\tilde{\pp}(t)\leq 0$ and hence $\left\| w_{1}- w_{2}\right\|_{\varSigma}=0$ . This proves uniqueness.
\end{proof}
\section{Dependence of Solutions}
\begin{theorem} [Dependence on initial conditions] Assume that the functions $ V,~U$ and $G$ satisfies the conditions (H1)--(H2). Let $w_{i}\,(i=1,2)$ be the mild solutions of the following problem corresponding to $ \varsigma_{i}(t),\,  (i=1,2)$,
\begin{small}
\begin{align} \label{e4.5}
w'(t)&=\mathscr{A} w(t)+V\left(t, w_t, \int_0^t U(t,s,w_s)ds  \right), ~t \in I=[0,b],~t \not= t_{k},~k=1,2,\cdots,m \\
w(t)& =\varsigma_{i}(t),\,  (i=1,2),~t \in [-r, 0], \label{e4.6}\\
\Delta w(t_{k})&= I_{k}\left( \int_{t_{k}-\tau_{k}}^{t_{k}-\theta_{k}}G(s,w_s)ds\right),~k=1,2,\cdots,m \label{e4.7}.
\end{align}
\end{small}
Then,
\begin{small}
\begin{align}\label{p1}
\left\| w_{1}- w_{2}\right\|_{\varSigma}\leq\mathscr {M} \left\| \varsigma_{1}-\varsigma_{2}\right\|_{\mathscr {C}} \underset{0<t_{k}<t}{\prod} ~C_{k} \exp \left (\int_{t_{\alpha}}^b \mathscr {M} N_{V}(s) \left[1+\int_0^s N_{U}(\sigma) d\sigma\right] ds\right)
\end{align}
\end{small}
where $C_{k}\, (k=1,2,\cdots,m)$ are given in \eqref{b4}.
\end{theorem}
\begin{proof}
Let any $\varsigma_{i}\in\mathscr{C}\,(i=1,2)$ and $ w_{i}\,(i=1,2)$ be the corresponding  mild solutions  problems \eqref{e4.5}--\eqref{e4.7}. Then, by hypotheses (H1) and (H2),
\begin{small}
\begin{align*}
&\left\| w_{1}(t)- w_{2}(t)\right\|\\
&\leq \left\| \mathscr{T}(t)\right\|_{B(X)} \left\| \varsigma_{1}-\varsigma_{2}\right\|_{\mathscr {C}} +\int_0^t\left\| \mathscr{T}(t-s)\right\|_{B(X)} N_{V}(s)\left\lbrace\left\|(w_{1}- w_{2})(s)\right\| _{\mathscr {C}} \right.\\
& \left. +\int_0^s N_{U}(\sigma)\left\| (w_{1}- w_{2})(\sigma)\right\| _{\mathscr {C}} d\sigma\right\rbrace ds +\underset{0<t_{k}<t}{\sum} \left\| \mathscr{T}(t-t_{k})\right\| _{B(X)} L_{G}\,\mathscr{D}_{k} \int_{t_{k}-\tau_{k}}^{t_{k}-\theta_{k}}\left\| (w_{1}- w_{2})(s)\right\| _{\mathscr {C}} ds,\,t\in [0,b].
\end{align*}
\end{small}
Consider the function $\tilde{\pp}$ defined by \eqref{3}. Then ,
\begin{small}
\begin{align}\label{e3.4}
\tilde{\pp}(t)& \leq \mathscr {M} \left\| \varsigma_{1}-\varsigma_{2}\right\|_{\mathscr {C}} +\int_0^t \mathscr {M} N_{V}(s)\left\lbrace \tilde{\pp}(s) +\int_0^s N_{U}(\sigma) \tilde{\pp}(\sigma) d\sigma\right\rbrace ds \nonumber\nonumber\\
&\qquad+\underset{0<t_{k}<t}{\sum} \mathscr {M} L_{G}\,\mathscr{D}_{k} \int_{t_{k}-\tau_{k}}^{t_{k}-\theta_{k}}  \tilde{\pp}(s)ds.
\end{align}
\end{small}
Applying the Pachpatte’s type integral inequality to \eqref{e3.4} with
\begin{small}
$$~u(t)= \tilde{\pp}(t),\,\, n(t)=\mathscr {M} \left\| \varsigma_{1}-\varsigma_{2}\right\|_{\mathscr {C}},~ f(t)=\mathscr {M} N_{V}(t) ,~g(t)=N_{U}(t),~\text{and}\, ~\beta_{k}= \mathscr {M} L_{G}\,\mathscr{D}_{k}$$
\end{small}
we get
\begin{small}
\begin{align*}
\tilde{\pp}(t)&\leq \mathscr {M} \left\| \varsigma_{1}-\varsigma_{2}\right\|_{\mathscr {C}} \underset{0<t_{k}<t}{\prod} ~C_{k} \exp \left (\int_{t_{\alpha}}^t \mathscr {M} N_{V}(s) \left[1+\int_0^s N_{U}(\sigma) d\sigma\right] ds\right)\\
&\leq \mathscr {M} \left\| \varsigma_{1}-\varsigma_{2}\right\|_{\mathscr {C}} \underset{0<t_{k}<t}{\prod} ~C_{k} \exp \left (\int_{t_{\alpha}}^b \mathscr {M} N_{V}(s) \left[1+\int_0^s N_{U}(\sigma) d\sigma\right] ds\right).
\end{align*}
\end{small}
Which gives the inequality \eqref{p1}.
\end{proof}

Consider the functional impulsive integrodifferential  equation involving the parameter of the form:
\begin{small}
\begin{align} \label{e4.1}
w'(t)&=\mathscr{A} w(t)+V\left(t,\varrho_{1}, w_t, \int_0^t U(t,s,w_s)ds \right), ~t \in I=[0,b],~t \not= t_{k},~k=1,2,\cdots,m, \\
\Delta w(t_{k})&= I_{k}\left( \int_{t_{k}-\tau_{k}}^{t_{k}-\theta_{k}}G(s,\mu_{1}, w_s)ds\right),~k=1,2,\cdots,m\label{e4.4}
\end{align}
\end{small}
and
\begin{small}
\begin{align}
w'(t)&=\mathscr{A} w(t)+V\left(t, \varrho_{2}, w_t, \int_0^t U(t,s,w_s)ds \right), ~t \in I=[0,b],~t \not= t_{k},~k=1,2,\cdots,m,  \label{e4.2}\\
\Delta w(t_{k})&= I_{k}\left( \int_{t_{k}-\tau_{k}}^{t_{k}-\theta_{k}}G(s,\mu_{2}, w_s)ds\right),~k=1,2,\cdots,m,\label{b5}
\end{align}
\end{small}
subject to
\begin{small}
\begin{align}
w(t)& =\varsigma(t), ~t \in [-r, 0],\label{e4.3}
\end{align}
\end{small}
where $\varrho_{1},\,~\varrho_{2},\, \mu_{1},\,\mu_{2}\in \mathbb{R}$, $V: I\times\mathbb{R}\times \mathscr {C}\times X \to X$ and $ G: I \times \mathbb{R}\times \mathscr {C} $ are given functions, $U \,\text{and}\,~ \varsigma $ are as in \eqref{e1.1}--\eqref{e1.3}.\\
Then we have  following  theorem  showing  dependence of solution
on parameters.
\begin{theorem}[Dependence on initial parameters] Let the hypotheses (H1)(iii) and (H2)  are holds. Suppose that
\begin{itemize}
\item[(H3)] there exist $ \tilde{N_{V}}\in C([0,b],\mathbb{R})$ and $ \varOmega_{\varrho},\varOmega_{1} >0 $ such that
\begin{small}
\begin{itemize}
\item[(i)] $\|V(t,\varrho,\vartheta_{1},w_{1})-V(t,\varrho,\vartheta_{2},w_{2})\| \leq \varOmega_{\varrho}\tilde{N_{V}}(t) \left(\|\vartheta_{1}-\vartheta_{2}\|_{\mathscr{C}}+\|w_{1}-w_{2}\|\right);$
\item[(ii)]$\|V(t,\varrho_{1},\vartheta_{1},w_{1})-V(t,\varrho_{2},\vartheta_{1},w_{1})\| \leq \varOmega_{1} \left| \varrho_{1}-\varrho_{2}\right|,$
\end{itemize}
\end{small}
for all $ t\in [0,b],~\varrho,\varrho_{1},\varrho_{2}\in \mathbb{R},~\vartheta_{1},\vartheta_{2}\in \mathscr{C}$ and $w_{1},w_{2}\in X $ and
\item[(H4)]there exist $\tilde{L_{G}}, \varOmega_{\mu},\varOmega_{2} >0 $ such that
\begin{small}
\begin{itemize}
\item[(i)] $\|G(t,\mu,\vartheta_{1})-G(t,\mu,\vartheta_{2})\| \leq \varOmega_{\mu}\tilde{L_{G}} \,\|\vartheta_{1}-\vartheta_{2}\|_{\mathscr{C}};$
\item[(ii)]$\|G(t,\mu_{1},\vartheta_{1})-G(t,\mu_{2},\vartheta_{1})\| \leq \varOmega_{2} \left| \mu_{1}-\mu_{2}\right|,$
\end{itemize}
\end{small}
\end{itemize}
for all $ t\in [0,b],~\mu,\mu_{1},\mu_{2}\in \mathbb{R}$ and $\vartheta_{1},\vartheta_{2}\in \mathscr{C}$.
 Then the mild solution  $w_{1}$ of \eqref{e4.1}--\eqref{e4.4} subject to \eqref{e4.3} and $w_{2}$ of  \eqref{e4.2}--\eqref{b5} subject to \eqref{e4.3}
satisfy the
\begin{small}
\begin{align}\label{p2}
& \left\| w_{1}- w_{2}\right\|_{\varSigma}\nonumber\\
&\leq \left( b\mathscr {M} \varOmega_{1} \left| \varrho_{1}-\varrho_{2}\right|+ \sum_{k=1}^{m} 2b\mathscr {M}\varOmega_{2}\mathscr{D}_{k} \left| \mu_{1}-\mu_{2}\right|\right)  \underset{0<t_{k}<t}{\prod} ~C_{k} \exp \left (\int_{t_{\alpha}}^b \mathscr {M}\tilde{ N_{V}}(s) \left[1+\int_0^s N_{U}(\sigma) d\sigma\right] ds\right),
\end{align}
\end{small}
where
\begin{small}
\begin{align*}
C_{k}&=\exp \left (\int_{t_{k-1}}^{t_{k}} \mathscr {M} \tilde{N_{V}}(s)\left[1+ \int_0^s N_{U}(\sigma)d\sigma\right] ds\right)\\
&\qquad +\mathscr {M}\tilde{ L_{G}}\,\mathscr{D}_{k} \int_{t_{k}-\tau_{k}}^{t_{k}-\theta_{k}} \exp \left (\int_{t_{k-1}}^{s} \mathscr {M} \tilde{N_{V}}(\sigma)\left[1+ \int_0^\sigma N_{U}(\xi)d\xi\right]d\sigma\right)ds.
\end{align*}
\end{small}
\end{theorem}
\begin{proof}
Let $w_{1}$ be the mild solution  of \eqref{e4.1}--\eqref{e4.4} subject to \eqref{e4.3}  and let $w_{2}$  be the  mild solution of \eqref{e4.2}--\eqref{b5} subject to \eqref{e4.3}. Then by using hypotheses (H1)(iii), (H2), (H3) and (H4) we have
\begin{small}
\begin{align*}
&\left\| w_{1}(t)- w_{2}(t)\right\|\\
&\leq\int_0^t \left\| \mathscr{T}(t-s)\right\|  \left\|  V\left(s, \varrho_{1}, (w_{1})_s, \int_0^s U(s,\sigma,(w_{1})_{\sigma})d\sigma \right)-V\left(s,\varrho_{2}, (w_{2})_s, \int_0^s U(s,\sigma,(w_{2})_{\sigma})d\sigma \right)\right\|  ds\\
& \qquad+\underset{0<t_{k}<t}{\sum} \left\| \mathscr{T}(t-t_{k})\right\| \left\|  I_{k}\left( \int_{t_{k}-\tau_{k}}^{t_{k}-\theta_{k}}G(s,\mu _{1},(w_{1})_s)\right) - I_{k}\left( \int_{t_{k}-\tau_{k}}^{t_{k}-\theta_{k}} G(s,\mu_{2},(w_{2})_s) ds\right)\right\| \\
&\leq\int_0^t \left\| \mathscr{T}(t-s)\right\|  \left\|  V\left(s, \varrho_{1}, (w_{1})_s, \int_0^s U(s,\sigma,(w_{1})_{\sigma})d\sigma \right)-V\left(s,\varrho_{1}, (w_{2})_s, \int_0^s U(s,\sigma,(w_{2})_{\sigma})d\sigma \right)\right\|  ds\\
& +\int_0^t \left\| \mathscr{T}(t-s)\right\|  \left\|  V\left(s, \varrho_{1}, (w_{2})_s, \int_0^s U(s,\sigma,(w_{2})_{\sigma})d\sigma \right)-V\left(s,\varrho_{2}, (w_{2})_s, \int_0^s U(s,\sigma,(w_{2})_{\sigma})d\sigma \right)\right\|  ds\\
& \qquad+\underset{0<t_{k}<t}{\sum} \left\| \mathscr{T}(t-t_{k})\right\| \left\|  I_{k}\left( \int_{t_{k}-\tau_{k}}^{t_{k}-\theta_{k}}G(s,\mu_{1},(w_{1})_s)\right) - I_{k}\left( \int_{t_{k}-\tau_{k}}^{t_{k}-\theta_{k}} G(s,\mu_{1},(w_{2})_s) ds\right)\right\| \\
& \qquad \qquad+\underset{0<t_{k}<t}{\sum} \left\| \mathscr{T}(t-t_{k})\right\| \left\|  I_{k}\left( \int_{t_{k}-\tau_{k}}^{t_{k}-\theta_{k}}G(s,\mu_{1},(w_{2})_s)\right) - I_{k}\left( \int_{t_{k}-\tau_{k}}^{t_{k}-\theta_{k}} G(s,\mu_{2},(w_{2})_s) ds\right)\right\| \\
&\leq b\,\mathscr {M} \varOmega_{1} \left| \varrho_{1}-\varrho_{2}\right|+ \sum_{k=1}^{m} 2\, b\, \mathscr {M}\,\varOmega_{2}\,\mathscr{D}_{k} \left| \mu_{1}-\mu_{2}\right|+\int_0^t\mathscr {M} \tilde{N_{V}}(s)\left\lbrace\left\|(w_{1}- w_{2})(s)\right\| _{\mathscr {C}}\right.\\
& \qquad \left. +\int_0^s N_{U}(\sigma)\left\| (w_{1}- w_{2})(\sigma)\right\| _{\mathscr {C}} d\sigma\right\rbrace ds
+\underset{0<t_{k}<t}{\sum} \mathscr {M} \tilde{L_{G}}\,\mathscr{D}_{k} \int_{t_{k}-\tau_{k}}^{t_{k}-\theta_{k}}\left\| (w_{1}- w_{2})(s)\right\| _{\mathscr {C}} ds.
\end{align*}
\end{small}
Then with the function $\tilde{\pp}$ given in \eqref{3}, we have,
\begin{small}
\begin{align}\label{e4.9}
\tilde{\pp}(t)
& \leq b\,\mathscr {M} \varOmega_{1} \left| \varrho_{1}-\varrho_{2}\right|+ \sum_{k=1}^{m} 2\, b\, \mathscr {M}\,\varOmega_{2}\,\mathscr{D}_{k} \left| \mu_{1}-\mu_{2}\right|+\int_0^t \mathscr {M} \tilde{N_{V}}(s)\left\lbrace\tilde{\pp}(t)
+\int_0^s N_{U}(\sigma)\tilde{\pp}(\sigma) d\sigma\right\rbrace ds\nonumber\\
&\qquad +\underset{0<t_{k}<t}{\sum} \mathscr {M} \tilde{L}_{G}\,\mathscr{D}_{k} \int_{t_{k}-\tau_{k}}^{t_{k}-\theta_{k}}\tilde{\pp}(s) ds.
\end{align}
\end{small}
Again, application of Pachpatte’s type integral inequality with
\begin{small}
$$~u(t)= \tilde{\pp}(t),\,\, n(t)=b\,\mathscr {M} \varOmega_{1} \left| \varrho_{1}-\varrho_{2}\right|+ \sum_{k=1}^{m} 2\, b\, \mathscr {M}\,\varOmega_{2}\,\mathscr{D}_{k} \left| \mu_{1}-\mu_{2}\right|,~ f(t)=\mathscr {M} \tilde{N_{V}}(t) ,~g(t)=N_{U}(t)$$
\end{small}
 $~\text{and}\, ~\beta_{k}= \mathscr {M}  \tilde{L_{G}}\,\mathscr{D}_{k}$ gives
\begin{small}
\begin{align*}
\pp(t)\leq \left( b\mathscr {M} \varOmega_{1} \left| \varrho_{1}-\varrho_{2}\right|+ \sum_{k=1}^{m} 2b\mathscr {M}\varOmega_{2}\mathscr{D}_{k} \left| \mu_{1}-\mu_{2}\right|\right)  \underset{0<t_{k}<t}{\prod} ~C_{k} \exp \left (\int_{t_{\alpha}}^b \mathscr {M}\tilde{ N_{V}}(s) \left[1+\int_0^s N_{U}(\sigma) d\sigma\right] ds\right).
\end{align*}
\end{small}
Which gives the inequality \eqref{p2}.
\end{proof}

Consider the following problem
\begin{small}
\begin{align} \label{e5.1}
v'(t)&=\mathscr{A} v(t)+\widehat{V}\left(t, v_t, \int_0^t U(t,s,v_s)ds  \right), ~t \in I=[0,b],~t \not= t_{k},~k=1,2,\cdots,m \\
v(t)& =\widehat{\varsigma}(t), ~t \in [-r, 0] \label{e5.2}\\
\Delta v(t_{k})&= \widehat{I_{k}}\left( \int_{t_{k}-\tau_{k}}^{t_{k}-\theta_{k}}G(s,v_s)ds\right),~k=1,2,\cdots,m\label{e5.3},
\end{align}
\end{small}
where  $\widehat{V}: I\times\times \mathscr {C}\times X \to X,\,\widehat{\varsigma}(t) \in \mathscr{C},\,\widehat{I_{k}}:X \to X,\, k=1, \cdots m$ are given functions. The functions  $U \,\text{and}\,~ G$ as in \eqref{e1.1}--\eqref{e1.3}.

The next theorem gives dependance of the solutions on the functions involved in the equations .
\begin{theorem}[Dependence on functions]\label{th5}
Let the hypotheses (H1)--(H2)  are holds. Suppose that
\begin{itemize}
\item[(H3)]there exists $\mathscr{P}, \,\mathscr{J},\mathscr{N}_{k}> 0$ such that
\begin{small}
\begin{itemize}
\item[(i)]$\|V(t,\vartheta_{1},w_{1})-\widehat{V}(t,\vartheta_{1},w_{1})\| \leq \mathscr{P} $;
\item[(ii)] $\left\| \varsigma(t)-\widehat{\varsigma}(t)\right\| \leq \mathscr{J}$;
\item[(iii)]$\left\| I_{k}(w_{1})- \widehat{I_{k}}(w_{1})\right\| \leq \mathscr{N}_{k}  ;\,k= 1, \cdots m $,
\end{itemize}
\end{small}
\end{itemize}
for all $ t \in I,~\vartheta_{1},\in \mathscr{C}$ and $w_{1}\in X$.
If $w$ and $v$ are the mild solutions of \eqref{e1.1}--\eqref{e1.3} and \eqref{e5.1}--\eqref{e5.3}
respectively then,
\begin{small}
\begin{align}\label{p3}
\left\| w- v\right\|_{\varSigma}\leq \left( \mathscr {M}\mathscr{J}+ b\,\mathscr {M}\,  \mathscr {P}+\sum_{K=1}^{m} \mathscr {M} \mathscr{N}_{k}\right) \underset{0<t_{k}<t}{\prod} ~C_{k} \exp \left (\int_{t_{\alpha}}^b \mathscr {M} N_{V}(s) \left[1+\int_0^s N_{U}(\sigma) d\sigma\right] ds\right),
\end{align}
\end{small}
where $C_{k}\, (k=1,2,\cdots,m)$ is given in \eqref{b4}.
\end{theorem}
\begin{proof}
Let  $w$ is the mild solution of  \eqref{e1.1}--\eqref{e1.3}  and  $v$ is the mild solution \eqref{e5.1}--\eqref{e5.3}. Then, for $t \in [0,b]$,
\begin{small}
\begin{align*}
&\left\| w(t)- v(t)\right\|\\
&\leq\left\| \mathscr{T}(t)\right\|_{B(X)} \left\|  \varsigma(0)-\widehat{\varsigma}(0)\right\|+\int_0^t \left\| \mathscr{T}(t-s)\right\|_{B(X)}   \left\|  V\left(s, w_{s}, \int_0^s U(s,\sigma,w_{\sigma})d\sigma \right)\right.\\
&\quad \left.-\widehat{V}\left(s,v_{s}, \int_0^s U(s,\sigma,v_{\sigma})d\sigma \right)\right\|  ds\\
& \qquad+\underset{0<t_{k}<t}{\sum} \left\| \mathscr{T}(t-t_{k})\right\|_{B(X)}  \left\|  I_{k}\left( \int_{t_{k}-\tau_{k}}^{t_{k}-\theta_{k}}G(s,w_s)ds\right)-\widehat{I_{k}}\left( \int_{t_{k}-\tau_{k}}^{t_{k}-\theta_{k}}G(s,v_s)ds\right)\right\|\\
&\leq\left\| \mathscr{T}(t)\right\|_{B(X)} \left\|  \varsigma(t)-\widehat{\varsigma}(t)\right\|+\int_0^t \left\| \mathscr{T}(t-s)\right\|_{B(X)}   \left\|  V\left(s, w_{s}, \int_0^s U(s,\sigma,w_{\sigma})d\sigma \right)\right.\\
&\quad \left.-V\left(s,v_{s}, \int_0^s U(s,\sigma,v_{\sigma})d\sigma \right)\right\|  ds+\int_0^t \left\| \mathscr{T}(t-s)\right\|_{B(X)}   \left\|  V\left(s, v_{s}, \int_0^s U(s,\sigma,v_{\sigma})d\sigma \right)\right.\\
&\qquad\quad \left.-\widehat{V}\left(s,v_{s}, \int_0^s U(s,\sigma,v_{\sigma})d\sigma \right)\right\|  ds\\
&\qquad \qquad  +\underset{0<t_{k}<t}{\sum} \left\| \mathscr{T}(t-t_{k})\right\|_{B(X)}  \left\|  I_{k}\left( \int_{t_{k}-\tau_{k}}^{t_{k}-\theta_{k}}G(s,w_s)ds\right)
-I_{k}\left( \int_{t_{k}-\tau_{k}}^{t_{k}-\theta_{k}}G(s,v_s)ds\right)\right\|\\
&\qquad \qquad \quad +\underset{0<t_{k}<t}{\sum} \left\| \mathscr{T}(t-t_{k})\right\|_{B(X)}  \left\|  I_{k}\left( \int_{t_{k}-\tau_{k}}^{t_{k}-\theta_{k}}G(s,v_s)ds\right)
-\widehat{I_{k}}\left( \int_{t_{k}-\tau_{k}}^{t_{k}-\theta_{k}}G(s,v_s)ds\right)\right\|\\
&\leq \mathscr {M}\mathscr{J}+\int_0^t \mathscr {M}  N_{V}(s)\left\lbrace\left\|w_{s}- v_{s}\right\| _{\mathscr {C}}
+\int_0^s N_{U}(\sigma)\left\| w_{\sigma}- v_{\sigma}\right\| _{\mathscr {C}} d\sigma\right\rbrace   ds + b  \mathscr {M}  \mathscr {P}\\
& \qquad +\underset{0<t_{k}<t}{\sum} \mathscr {M} L_{G}\,\mathscr{D}_{k} \int_{t_{k}-\tau_{k}}^{t_{k}-\theta_{k}}\left\| w_{s}- v_{s}\right\| _{\mathscr {C}} ds +\sum_{K=1}^{m}\mathscr {M} \mathscr{N}_{k}.
\end{align*}
\end{small}
With the function $\tilde{\pp}$ in \eqref{3}, we have
\begin{small}
\begin{align}\label{e5.4}
&\tilde{\pp}(t)\leq\mathscr {M}\mathscr{J}+ b  \mathscr {M}  \mathscr {P}+\sum_{K=1}^{m} \mathscr {M} \mathscr{N}_{k} +\int_0^t \mathscr {M}  N_{V}(s) \left\lbrace\tilde{\pp}(s)+\int_0^s N_{U}(\sigma)\tilde{\pp}(\sigma) d\sigma\right\rbrace ds\nonumber\\
&\qquad \quad+\underset{0<t_{k}<t}{\sum} \mathscr {M} L_{G}\,\mathscr{D}_{k} \int_{t_{k}-\tau_{k}}^{t_{k}-\theta_{k}} \tilde{\pp}(s) ds.
\end{align}
\end{small}
By applying the Pachpatte’s type integral inequalities
we obtain
\begin{small}
\begin{align*}
\tilde{\pp}(t)\leq \left( \mathscr {M}\mathscr{J}+ b\,\mathscr {M}\,  \mathscr {P}+\sum_{K=1}^{m} \mathscr {M} \mathscr{N}_{k}\right)  \underset{0<t_{k}<t}{\prod} ~C_{k} \exp \left (\int_{t_{\alpha}}^b \mathscr {M} N_{V}(s) \left[1+\int_0^s N_{U}(\sigma) d\sigma\right] ds\right).
\end{align*}
\end{small}
Which gives the desired inequality \eqref{p3}.
\end{proof}
\section{Examples}
Consider the nonliner Volterra delay integrodifferntial equations with integral impluses
\begin{small}
\begin{align} 
 w'(t)-w(t)&=1-\sin(5)-\sin(w(t-5))+  \int_0^t [1+\cos(w(s-5))]ds, ~t \in I=[0,2],~t \not= t_{1}=1,\label{5.1} \\
w(t)& =t, ~t \in [-1, 0],\label{5.2}\\
\Delta w(t_{1})&= \sin\left( \int_{t_{1}-\tau_{1}}^{t_{1}-\theta_{1}}G(s,w(t-5))ds\right),\label{5.3}
\end{align}
\end{small}
where $G$ is any  Lipschitz continuous function with Lipschitz constant $L_{G}$.

Consider the real Banach space $ (\mathbb{R},\left|\cdot\right| )$ and
define $\mathscr{T}(t): [0,\infty) \to \mathcal{B}(\mathbb{R}) $ by $\mathscr{T}(t)x=e^tx,\, ~x\in \mathbb{R}.$
Then  $\{\mathscr{T}(t)\}_{t\geq 0}$ is semigroup of bounded linear operators in $ \mathbb{R}.$ with infinitesimal generator  $\mathscr{A}=I=\text{Indentity operator on}\,\mathbb{R} $. 

For the choices $ \theta_{1}=\tau_{1}=\frac{1}{2}$  and the function $I_1:\mathbb{R}\to \mathbb{R}$ defined by $I_1(x)= \sin x, x \in \mathbb{R}$, the problem \eqref{5.1}--\eqref{5.3} can be written as
\begin{small}
\begin{align} 
&w'(t)
=\mathscr{A}w(t)+V\left(t, w(t-5), \int_0^t U(t,s,w(s-5))ds \right)\label{5.4} \\
& w(t) =t, ~t \in [-1, 0],\label{5.5}\\
& \Delta w(t_{1})= 0,\label{5.6}
\end{align}
\end{small} 
where $U:[0,2]\times[0,2]\times\mathbb{R}\to \mathbb{R} $ is defined by
 $$ U\left( t, s, w\right) = -1+ \cos w $$
and $V:[0,2] \times \mathbb{R} \times \mathbb{R}\to \mathbb{R} $
is defined by
$$
V\left(t, v, w\right)=1-\sin(5)-\sin v +  w .
$$
Note that 
$$
 \left| U(t,s ,w_{1})-U(t,s ,v_{1})\right|\leq\ \left| \cos w_{1}-\cos v_{1}\right|\leq\ \left| w_{1}-v_{1}\right|,
$$
 and 
$$
  \left| V(t,w_{1},w_{2})-V(t,v_{1},v_{2})\right|\leq\ \left| \sin w_{1}-\sin v_{1}\right|+ | w_{2}-v_{2}| \leq\ | w_{1}-v_{1}|+| w_{2}-v_{2}|,
$$
for any $ t, s \in [0,2]$ and $w_{1}, w_{2}, v_{1}, v_{2}\in \mathbb{R}$.

In this case, the condition \eqref{c1} takes the form
$$2\,b\mathscr {M} L_{G}\,\mathscr{D}_{1}= 4 \mathscr {M} L_{G} <1$$
If we choose $L_{G}$ such that $0<L_{G}<\frac{1}{4 \mathscr {M}}$ then by applying Theorem \ref{th3.3} the problem \eqref{5.1}--\eqref{5.3} has at least one solution. 

\section*{Acknowledgement}
The second author is financially supported by UGC, New Delhi, India. Ref: F1-17.1/2017-18/RGNF-2017-18-SC-MAH-43083.

  \end{document}